\documentclass{article}

\usepackage[french, english]{babel}
\usepackage[T1]{fontenc} 
\usepackage[utf8]{inputenc} 
\usepackage{amsmath,amsthm,amsfonts,amscd,amssymb,graphics,color,xspace}  
\usepackage{xr} 
\usepackage[a4paper,top=2cm,bottom=2cm,left=3cm,right=3cm,marginparwidth=1.75cm]{geometry}
\usepackage{tikz-cd} 
\usepackage{amsmath,amsthm,amsfonts,amscd,amssymb,graphics,color,xspace}
\usepackage{amsmath}
\usepackage{graphicx}
\usepackage[colorlinks=true, allcolors=blue]{hyperref}
\usepackage{enumitem}
\usepackage{cite} 
\usepackage{xr}
\title{Tauberian Theorem: Square-root singularity and Ramified covering of degree two}
\author{Chevalier Guillaume.}

\begin{document}
\newtheorem{Thm}{Theorem}[section]
\newtheorem{Prop}[Thm]{Proposition} 
\newtheorem{Lem}[Thm]{Lemma} 
\newtheorem{Cor}[Thm]{Corollary} 
\newtheorem{Clm}[Thm]{Claim} 
\newtheorem{prop}[Thm]{Property} 
\newtheorem*{Thm*}{Theorem} 
\newtheorem*{Cor*}{Corollary} 
\newtheorem*{Prop*}{Proposition}
\theoremstyle{definition} 
\newtheorem{Def}[Thm]{Definition} 
\theoremstyle{remark}
\newtheorem{Rem}[Thm]{Remark} 
\newtheorem{Notation}[Thm]{Notation}
\newtheorem{Ex}[Thm]{Example}  
\newtheorem*{Rem*}{Remark}
\newtheorem*{Ex*}{Example}
\numberwithin{equation}{section}

\newcounter{nmbrexercise}
\newenvironment{Exercise}[1][ ] 
{
  \par\vspace{\baselineskip}%
 {\refstepcounter{nmbrexercise} \color{PineGreen} \noindent \textbf{Exercise~\thenmbrexercise :} \textit{#1}}%
  \par\vspace{\baselineskip}%
}%
{\vspace{\baselineskip}}
\maketitle

\begin{abstract}
We prove a Tauberian theorem concerning power series admitting square root singularities. More precisely we give an asymptotic expansion to any order of the coefficients of a power series admitting square-root type singularities. This article is one of a triptych with \cite{Asymp} and \cite{LCurve} that aims at proving an asympototic expansion to any order of the passage probability of an irreducible finite range random walk on free groups.
\end{abstract}

\begin{center}
	\textbf{\Large Introduction}
\end{center}

        Historically, the term \textit{Tauberian Theorem} originates from Alfred Tauber (1866–1942), who established a well-known result concerning the Abel summability criterion \cite{Tauber_1897}. Subsequent developments extended this idea to the extraction of coefficient asymptotics from the behaviour of the associated analytic function, near its singularities. Notable examples include Karamata’s Tauberian theorem \cite{Karamata_1933}, elegantly presented in William Feller’s book \cite{Feller_Vol_II_1971} (XIII.5, Theorem 1), which provides an asymptotic expansion of the tail of a real measure based on the behavior of its Laplace transform near zero, with application for generating functions of sequence of non-negative numbers. Another important result is the Ikehara-Wiener theorem \cite{Ikehara_1931}, which establishes asymptotics for partial sums of a sequence $(a_n)_n$ in terms of the behavior of its associated Dirichlet series.

    	In the book \underline{Analytic Combinatorics} by Philippe Flajolet and Robert Sedgewick \cite{Flajolet-Sedgewick_2009}, one can find a collection of methods and Tauberian theorems for computing asymptotic expansions of the coefficients of generating functions that admit a power series expansion in a neighbourhood of $0\in\mathbb{C}$ depending on the behaviour of the generating function near its nearest singularities. among these results, one can find a theorem (Theorem VII.6 in subsection VII.6.3 of \cite{Flajolet-Sedgewick_2009}) whose proof uses contour integral as we does here to prove Theorem \ref{Thm_Tauberien_tout_ordre} and Proposition \ref{Prop_Composition_racine}, below. 

	    We present our Tauberian Theorem using a different point of view than the one in Flajolet and Sedgewick book, to address rational-type singularities, with a special focus on square-root type singularities. The main results of this article are Theorem \ref{Thm_Black_Box_aperiodic} and its Corollary \ref{Cor_Black_box_general}. That we will use in subsequent article (\cite{Asymp} and \cite{LCurve}) to establish a precise asymptotic expansion for irreducible finite-range random walk on free group. Additionally, we study the singularities of functions that factors with some meromorphic functions on a Riemann surface $\mathcal{C}$. 

		
%
%

\section{Tauberian Theorem}
\subsection{Stirling formula The Historical Formula (from \texorpdfstring{\cite{Dominici_2008})}{-CitationDominici}}
		\begin{Def}\textit{Power functions on $\mathbb{C}$}\label{def_power_functions}

	Let us note $L:\mathbb{C}\setminus ]-\infty,0]\to\mathbb{C}$, the main determination of the logarithm on $\mathbb{C}\setminus ]-\infty,0]$ (i.e. the one that verifies $L(1)=0$).
	For $\beta\in\mathbb{C}$, note $\pi_\beta:z\mapsto\exp (\beta L(z))$ defined on $\mathbb{C}\setminus ]-\infty,0]$. We will also write $\pi_\beta(z)=z^\beta$.

	If $Re(\beta)>0$, $\pi_\beta$ is continuously extended at $z=0$ by $\pi_\beta(0)=0$.
\end{Def}

			During his exchanges with the French mathematician Abraham de Moivre (1667-1754), the Scottish mathematician James Stirling (1692-1770) demonstrated the following result about an equivalent of $n!$, when $n$ goes to $\infty$.
			\begin{Prop}\textit{Stirling's formula (1730)}(Ex 2 to Prop 28 \cite{Tweddle_2003})
			
				The factorial sequence $(n!)_{n\in\mathbb{N}}$ satisfies the following equivalence as $n$ goes to $\infty$:
				\begin{equation}\label{Stirling's_formula}
				n!\sim\sqrt{2 \pi n}\left(\frac{n}{e}\right)^n
				\end{equation}
			\end{Prop}

This formula (\ref{Stirling's_formula}) can be generalized using Euler's Gamma function:
			\begin{Prop}\textit{Euler Gamma function }(See\cite{Flajolet-Sedgewick_2009} Appendix B.3.)
			
			    The Euler Gamma function is the function $\Gamma:\{Re(z)>0\}\to\mathbb{C}$, that associates to any complex number $z\in\mathbb{C}$, with positive real part $Re(z)>0$, the value $$\Gamma(z)=\int_0^\infty t^{z-1}e^{-t}dt.$$
			   
			   It is holomorphic on $\{z: Re(z)>0\}\subset\mathbb{C}$ and extends by means of the formula \og $\Gamma(z+1)=z\Gamma(z)$\fg{} into a meromorphic function on $\mathbb{C}$ admitting simple poles at all non-positive integers $z=0,-1,-2,...$, with residue $(-1)^n/n!$ at $z=-n$, for $n\in\mathbb{N}$.
			\end{Prop}

			\begin{Prop}\textit{Generalized Stirling's formula}

				For $x\in\mathbb{R}$, when $x$ goes to $\infty$,
				$$\Gamma(1+x)\sim \sqrt{2\pi x}\left(\frac{x}{e}\right)^x$$
			\end{Prop}
            
            This asymptotic can be improved: 
			\begin{Prop}\textit{Laplace formula}

				There exist constants $c_1,c_2,...$ such that as $x$ goes to $\infty$, for any non-negative integer $K$, we have
				$$\Gamma(x+1)=\sqrt{2\pi x}\left(\frac{x}{e}\right)^x \left(1+\sum_{l=1}^{K-1} \frac{c_l}{x^l} + O\left(\frac{1}{x^K}\right)\right),$$
				where the constants $c_l$ are calculable (See\cite{Comtet_1974} p.267 *23)\footnote{$c_l$ is a rational number whose numerator is given by the sequence \href{https://oeis.org/A001163}{OEI\_A001163} and whose denominator is given by the sequence \href{https://oeis.org/A001164}{OEI\_A001164}.}.

				The first constants are given by:
				$$ c_1=\frac{1}{12}; \quad c_2= \frac{1}{288}; \quad c_3=- \frac{139}{51840}; \quad c_4= -\frac{571}{2488320 }; \;...$$
			\end{Prop}
			There are many formulas for approximating the Gamma function. Most of them are numerically more efficient than Laplace's formula. See \cite{Luschny_2016} for a benchmark of the various formulas and algorithms for approximating the factorial function. The Chinese mathematician Weiping Wang in \cite{Wang_2016} proposed a unified approach to all these different formulas. They are all of the form.
			\begin{equation}
				\Gamma(x+1)\sim\sqrt{2\pi x} \left(\frac{x}{e}\right)^x \frac{1}{e^a}\left(1 + \frac{b}{x}\right)^{cx +d}\left(\sum_{n=0}^\infty \frac{\alpha_n}{(x+h)^n}\right)^{\frac{x^l}{r}+q}
			\end{equation}
			with $a,b,c,d,h,l,q,r$ non-negative real numbers, $r>0$. In our case we will only use Laplace's formula, which corresponds to the case: $a=b=h=l=q=0$ and $r=1$.
			
			In the rest of this paper we will use the Tricomi-Erdelyi's formula \footnote{See\cite{Erdelyi-Tricomi_1951} or \cite{Luke_Vol_I_1969} p.33, for the original proof. See also\cite{Buric-Elezovic_2011} for a study of the constants \og $e_i$\fg{} .}:
			\begin{Prop}[Erdelyi-Tricomi (1951)]
			Let $a,b$ be two fixed complex numbers. There exist constants $(e_l)_{l\geq 1}$ such that for any non-negative integer $K\geq 0$ and for any $\varepsilon>0$, we have the equality
			\begin{equation}\label{Tricomi-Erdelyi_formula}
			    \frac{\Gamma(z+a)}{\Gamma(z+b)}=z^{a-b}\left(1+\sum_{l=1}^{K-1} \frac{e_k}{z^k} + O\left(z^{-K}\right)\right).
			\end{equation}
			satisfied for any complex number $z$ such that $|Arg(z+a)|<\pi-\varepsilon$.
			\end{Prop}
			\begin{Rem}
	The constants $e_1,e_2,...$ in the previous proposition are computable (See\cite{Erdelyi-Tricomi_1951},\cite{Luke_Vol_I_1969} p.33 or\cite{Buric-Elezovic_2011}): For any integer $l\geq 1$, we have
	\begin{equation}
		e_l=(-1)^l\frac{(a-b+1)(a-b+2)\cdot\cdot (a-b+l)}{l!}B_l^{(a-b+1)}(a)
	\end{equation}
	where\og $B_l^{(\beta)}$\fg{} for $\beta\in\mathbb{C}$ is the generalized Bernoulli polynomial function defined by the identification for $|t|<2\pi$
	\begin{equation}
		\frac{t^{\beta}e^{z t}}{(e^t-1)^{\beta}}=\sum_{l=0}^\infty \frac{t^l}{k!}B_l^{(\beta)}(z).
	\end{equation}
	
\end{Rem}
\subsection{Asymptotic Expansion to any order for square-root singularities}

\begin{Lem}\label{Decomposition_Lemma}\textit{Decomposition Lemma}

	Let $h$ be a holomorphic function on a domain $\Omega\subset\mathbb{C}$. For any $z_0\in\Omega$, and for any integer $m\in\mathbb{N}$, there exists a holomorphic function $\tilde{h}_m$ on $\Omega$, such that:
	\begin{gather*}
		\forall z\in\Omega,\; h(z)=\sum_{k=0}^{m-1}\frac{h^{(k)}(z_0)}{k!} (z-z_0)^k + (z-z_0)^m\tilde{h}_m(z)\\
		\text{and}\\
		\tilde{h}_m(z_0)=\frac{h^{(m)}(z_0)}{m!},
	\end{gather*}
	where $h^{(k)}$ is the notation taken for the $k$-th derivative of $h$.
\end{Lem}
\begin{proof}
	All we need to do is check that the function $h_m:\Omega\to\mathbb{C}$ such that for any complex number $z\in\Omega$,
	$$ h_m(z)= h(z)-h(z_0)-h'(z_0)(z-z_0)-...-\frac{h^{(m-1)}(z_0)}{(m-1)!}(z-z_0)^{m-1}$$
	admits a zero of order at least $m$ at $z=z_0$. This is immediate.
\end{proof}

\begin{Lem}\label{Root_lemma}~

	Let $\beta\notin\mathbb{N}$ be a real number. Let $\sum_{n\in\mathbb{N}} c_n(\beta) z^n$ be the power series expansion of the holomorphic map $z\mapsto (1-z)^\beta$, in the neighbourhood of $z=0$. 
	Then, when $n\in\mathbb{N}$ goes to $\infty$, we have the equivalent:
	$$ c_n(\beta)\sim - \frac{\beta}{\Gamma(1-\beta)}\frac{1}{n^{\beta+1}}$$
	where $\Gamma$ is Euler's gamma function.
\end{Lem}
\begin{proof}
	The Taylor formula for the coefficients of the Taylor series expansion in the neighbourhood of $z=0$ of $z\mapsto (1-z)^\beta$ gives the values:
	\begin{equation}\label{eq-t1}
		c_n(\beta)=
		\left\lbrace
			\begin{matrix}
				&1, \text{if } n=0\\
				&(-1)^n\frac{\beta(\beta-1)...(\beta-n+1)}{n!}, \text{if } n\geq 1
			\end{matrix}
		\right.
	\end{equation}
	So for $n\geq 1$, $c_n(\beta)=\dfrac{-\beta}{n!}(1-\beta)\cdot\cdot\cdot(n-1-\beta)=-\frac{\beta}{\Gamma(n+1)}\frac{\Gamma(n-\beta)}{\Gamma(1-\beta)}$ where we used the following property to simplify (\ref{eq-t1}): 
	$$\forall z: Re(z)>0,\, \Gamma(z+1)=z\Gamma(z).$$
	Finally, we apply Stirling's formula (\ref{Stirling's_formula}) to obtain when $n$ tends to $\infty$:
	\begin{align*}
		c_n(\beta)&\sim\frac{-\beta}{\Gamma(1-\beta)}\frac{\sqrt{2\pi(n-\beta+1)}}{\sqrt{2\pi n}}
		\left(\frac{n-\beta-1}{e}\right)^{n-\beta-1}
		\left(\frac{e}{n}\right)^n\\
		&\sim \frac{-\beta}{\Gamma(1-\beta)}\frac{1}{n^{\beta+1}}e^{\beta+1}
		\underbrace{\left(1-\frac{\beta+1}{n}\right)^{n-\beta-1}}_{\rightarrow e^{-(\beta+1)}}.
	\end{align*}
	Hence:
	$$c_n(\beta)\sim -\frac{\beta}{\Gamma(1-\beta)}\frac{1}{n^{\beta+1}}$$
\end{proof}
By using a stronger approximation than the Stirling formula, namely the Tricomi-Erdelyi formula (\ref{Tricomi-Erdelyi_formula}) to estimate the quotient $\frac{\Gamma(n-\beta)}{\Gamma(n+1)}$, in the previous proof, we can improve the conclusion of the lemma for free:
\begin{Lem}\label{Advanced_Root_Lemma}~

	Let $\beta\notin\mathbb{N}$ be a real number. Let $\sum_{n\in\mathbb{N}} c_n(\beta) z^n$ be the integer series expansion of the holomorphic map $z\mapsto (1-z)^\beta$ in the neighbourhood of $z=0$. 
	There are constants $e_1,e_2,...$ such that for any integer $m\in\mathbb{N}$ and any $n\in\mathbb{N}\setminus\{0\}$:
	$$c_n(\beta)=-\frac{\beta}{\Gamma(1-\beta)}\frac{1}{n^{\beta+1}}\left( 1 + \frac{e_1}{n} + \frac{e_2}{n^2} + ... + \frac{e_m}{n^m}+O_{n\to\infty} \left(\frac{1}{n^{m+1}}\right)\right)$$
\end{Lem}
\begin{proof}
	Let us resume the previous proof. We had obtained that
	$$c_n(\beta)=-\frac{\beta}{\Gamma(1-\beta)}\frac{\Gamma(n-\beta)}{\Gamma(n+1)}$$ 
	By Tricomi-Erdelyi's formula (\ref{Tricomi-Erdelyi_formula}),
	there exist constants $e_1,e_2,...$ that only depend on $\beta$ such that for any $n\in\mathbb{N}\setminus\{0\}$:
	\begin{equation}
		\frac{\Gamma(n-\beta)}{\Gamma(n+1)}=\frac{1}{n^{\beta+1}}\left(1 + \frac{e_1}{n} + \frac{e_2}{n^2} + ... + \frac{e_m}{n^m} + O_{n\to\infty}\left(\frac{1}{n^{m+1}}\right)\right)
	\end{equation} 
	The result follows.
\end{proof}

\begin{Lem}\label{Lemma_d_endadrement_Fourrier}~

	Let $\varphi:]0,2\pi[\to\mathbb{C}$ be a $\mathcal{C}^1$ function of integrable derivative, then for $n\in\mathbb{N}\setminus\{0\}$, if $c_n(\varphi)$ denotes the $n$-th Fourier coefficient of the function $\varphi$ then we have:
	$$c_n(\varphi)=\frac{1}{i n}( c_n(\varphi')- c_0(\varphi'))$$
	In particular, if $\displaystyle\lim_{\theta\to 0^+}\varphi(\theta)=\displaystyle\lim_{\theta\to 2\pi^-}\varphi(\theta)$, then $c_0(\varphi')=0$ and we immediately have $|c_n(\varphi)|\leq\frac{1}{n}||\varphi'||_{\mathcal{L}^1}$, with $||f||_{\mathcal{L}^1}=\frac{1}{2\pi} \int_0^{2\pi} |f(\theta)|d\theta$ for any integrable function $f$ on $]0,2\pi[$.
\end{Lem}
\begin{proof}
	Recall that $c_n(\varphi)=\frac{1}{2\pi}\int_{0}^{2\pi} \varphi(\theta)e^{-in\theta} d\theta$. Simply perform an integration by parts on this integral to obtain the result.
\end{proof}

\begin{Thm}\textit{Tauberian theorem: Asymptotic expansion to the first order for square root singularities}\label{Thm_singularite_fractional}

	Let $g$ be a continuous function on $\overline{\mathbb{D}}$ whose restriction to $\mathbb{D}$ is holomorphic. Let $\displaystyle\sum_{n\in\mathbb{N}} a_n z^n$ be the power series expansion of $g$ in the neighbourhood of $z=0$.
	 Suppose there exists a function $h$ that is holomorphic in a neighbourhood of the closure of the set $\mathbb{D}(1,1)^{1/2}=\{\omega^{1/2}: |\omega-1|<1\}$ in $\mathbb{C}$, denoted $\overline{\mathbb{D}(1,1)^{1/2}}$, such that
	 \begin{equation*}
	 	\forall z\in\overline{\mathbb{D}},\, g(z)=h(\sqrt{1-z})
	\end{equation*} 
	Then we have the asymptotic expansion:
	\begin{equation}\label{eq-t24}
		a_n=-\frac{h'(0)}{2\sqrt{\pi}}\frac{1}{n^{3/2}} + o_{n\to\infty}\left(\frac{1}{n^{3/2}}\right)
	\end{equation}
\end{Thm}

\begin{proof}
	With the notations of the theorem. For $\beta\in\mathbb{R}$, $\theta\in ]0,2\pi[$, we set 
	$$p_\beta(\theta):=(1-e^{i \theta})^\beta.$$ According to Cauchy's formula\footnote{Cf. \cite{Queffelec_2017} Theorem 3.2 p.131} applied to $g$ which is continuous on $\overline{\mathbb{D}}$ and holomorphic on $\mathbb{D}$, the coefficients $a_n$ verify:
	$$a_n=\frac{1}{2i\pi}\int_{\partial\mathbb{D}}\frac{g(z)}{z^{n+1}}dz.$$
	Set $\alpha:=1/2$. Parametrizing the disk boundary $\partial\mathbb{D}$ by the function $\theta\mapsto e^{i\theta}$ defined on $]0,2\pi[$, we obtain:
	\begin{equation*}
		a_n=\frac{1}{2\pi}\int_0^{2\pi}g(e^{i\theta})e^{-in\theta}d\theta=\frac{1}{2\pi}\int_0^{2\pi}h(p_\alpha(\theta))e^{-in\theta}d\theta
	\end{equation*}

	Let $m\in\mathbb{N}$ be such that $\alpha m-2>0$ and let $\tilde{h}_m$ be the holomorphic function in a neighbourhood of $\overline{\mathbb{D}(1,1)^\alpha}$ coming from the Lemma \ref{Decomposition_Lemma} with $z_0=0$. The function $\tilde{h}_m$ verifies for any complex number $z$ in this neighbourhood,
	\begin{gather*}
		h(z)=\sum_{k=0}^{m-1} \frac{h^{(k)}(0)}{k!}  z^k + z^m\tilde{h}_m(z)\\
 		\text{and}\\
 		\tilde{h}_m(0)=\frac{h^{(m)}(0)}{m!}
	\end{gather*}
	Then,
	\begin{equation}\label{formula_of_a_n}
		a_n=\sum_{k=0}^{m-1}  \frac{h^{(k)}(0)}{k!}\left(\frac{1}{2\pi}\int_0^{2\pi} p_\alpha(\theta)^k e^{-in\theta} d\theta\right)
		+ \frac{1}{2\pi}\int_0^{2\pi} \tilde{h}_m(p_\alpha(\theta))p_\alpha(\theta)^m e^{-in\theta} d\theta
	\end{equation}
	Note that for any integer $k\in\mathbb{N}$, $p_\alpha(\theta)^k=p_{\alpha k}(\theta)$ and by the Lemma \ref{Root_lemma}, we already have estimates for the integrals $c_n(\alpha k)=\displaystyle \frac{1}{2\pi}\int_0^{2\pi} p_{\alpha k}(\theta) e^{-in\theta} d_\theta$. For any positive integer $n\in\mathbb{N}\setminus\{0\}$, we have:
	\begin{equation}
		\begin{matrix}\label{formula_of_c_n_and_a_n}
			k=0,& c_n(0)=\delta_0(n); \\
			k=1,&\,
			c_n(\alpha)=-\dfrac{\alpha}{\Gamma(1-\alpha)}\dfrac{1}{n^{\alpha+1}}+o_{n\to\infty}\left(\dfrac{1}{n^{\alpha+1}}\right);\\
			m-1\geq k>1,&\,c_n(\alpha k)=O_{n\to\infty}\left(\dfrac{1}{n^{k\alpha+1}}\right).&
		\end{matrix}
	\end{equation}
	Then
	$$ \frac{1}{2\pi}\int_0^{2\pi} \tilde{h}_m(p_\alpha(\theta))p_\alpha(\theta)^m e^{-in\theta} d_\theta = c_n\left((\tilde{h}_m\circ p_\alpha)\times p_{\alpha m}\right).$$
	The function $\varphi$ defined for $\theta\in[0,2\pi]$ by $\varphi(\theta)=\tilde{h}_m(p_\alpha(\theta))p_{\alpha m}(\theta)$ is two times differentiable, and its second derivative is integrable on $]0,2\pi[$. Indeed, remark that for any real number $\beta$, the function $p_\beta$ is integrable on $]0,2\pi[$ if, and only if $\beta>-1$ and is derivable, of derivative $p_\beta'(\theta)=(-ie^{i\theta})\beta p_{\beta-1}(\theta)$. 
	Using the Leibniz formula we have that the $l$-th derivative of $\varphi$ is equal to a sum of functions of the form 
	\begin{equation}\label{eq-t23}
		Ce^{i k \theta} \tilde{h}_m^{(l_1)}(p_\alpha(\theta))p_{(\alpha(m+l_1)-l_1-l_2)}(\theta),
	\end{equation}

	with $k\leq l$ and $l_1+l_2\leq l$. Functions of this forms are integrable over $]0,2\pi[$ as long as we have $\alpha(m+l_1)-l_1-l_2>-1$. Since $l_1+l_2\leq l$ and $\alpha\geq 0$, we have $\alpha(m+l_1)-l_1-l_2\geq \alpha m - l>-1$, which is satisfied for $l\leq 2$, by choice of $m$.

Note also that for any $\beta > 0$, we have \begin{equation}
	\lim_{\theta\to 0} p_\beta(\theta)=\lim_{\theta\to 2\pi}p_\beta(\theta).
\end{equation} From the inequality $\alpha m - l>0$, for $l\leq 2$,  the same analogous limit holds with functions (\ref{eq-t23}). We thus deduce: 
$$\lim_{\theta\to 0} \varphi^{(l)}(\theta)=\lim_{\theta\to 2\pi}\varphi^{(l)}(\theta)=0,$$ for any $l\leq 2$. And in particular $c_0(\varphi')=c_0(\varphi'')=0$
	 
	 Then, applying two times the Lemma  \ref{Lemma_d_endadrement_Fourrier}, to $\varphi$, then to $\varphi'$, knowing that $c_0(\varphi')=c_0(\varphi'')=0$, we obtain for $n\geq 1$:
	
	\begin{equation}\label{formula_of_the_rest_of_a_n}
		\frac{1}{2\pi}\int_0^{2\pi} \tilde{h}_m(p_\alpha(\theta))p_{\alpha m}(\theta) e^{in\theta} d_\theta =O_{n\to\infty}\left(\frac{1}{n^{2}}\right).
	\end{equation}
	Finally, considering that for any $m-1\geq k >1$, $O_{n\to\infty}\left(\dfrac{1}{n^{k\alpha+1}}\right)=o\left(\frac{1}{n^{\alpha+1}}\right)$, and that $O\left(\frac{1}{n^2}\right)=o\left(\frac{1}{n^{\alpha+1}}\right)$, by introducing (\ref{formula_of_c_n_and_a_n}) and (\ref{formula_of_the_rest_of_a_n}) into the formula of $a_n$ (\ref{formula_of_a_n}) we obtain for any integer $n\geq 1$:
	\begin{equation}\label{eq-t25}
		\begin{split}
		a_n&= h(0)\delta_0(n) - \frac{\alpha}{\Gamma(1-\alpha)}\frac{h'(0)}{n^{\alpha+1}}+o\left(\frac{1}{n^{\alpha+1}}\right) \\
		&=-\frac{\alpha}{\Gamma(1-\alpha)}\frac{h'(0)}{n^{\alpha+1}}+o\left(\frac{1}{n^{\alpha+1}}\right)
		\end{split}
	\end{equation}
	Remembering that $\alpha=1/2$ and $\Gamma(1/2)=\sqrt{\pi}$ we get the desired formula (\ref{eq-t24}).
\end{proof}

\begin{Rem}
	 Formula (\ref{eq-t25}) actually holds for any $\alpha\in ]0,1[$ whenever we are interested on a holomorphic function $g$ that has the form $g(z)=h((1-z)^\alpha)=\sum_n a_n z^n$, with $h$ being a holomorphic function in a neighbourhood of the compact set $\overline{\{(1-z)^{\alpha}: z\in\mathbb{D}(0,1)\}}$.
\end{Rem}
In view of the Lemma \ref{Advanced_Root_Lemma}, we can improve this theorem by giving the following orders in the asymptotic development of the coefficients \og $a_n$\fg{}.

\begin{Thm}\label{Thm_Tauberien_tout_ordre}\textit{Tauberian theorem: Asymptotic expansion to any order for square root singularities}

	Let $g$ be a continuous function on $\overline{\mathbb{D}}$ whose restriction to $\mathbb{D}$ is holomorphic. Let $\displaystyle\sum_{n\in\mathbb{N}} a_n z^n$ be the power series expansion of $g$ in a neighbourhood of $0$.
 	Suppose there exists a function $h$ that is holomorphic in a neighbourhood of $\overline{\mathbb{D}(1,1)^{1/2}}\subset\mathbb{C}$ such that
 	\begin{gather*}
 		\forall z\in\overline{\mathbb{D}},\, g(z)=h(\sqrt{1-z})
	\end{gather*}
	Then there exist a unique sequence of constants $(c_l)_{l\geq 1}$ and a constant $C$ such that for any positive integer $K\in\mathbb{N}\setminus\lbrace 0\rbrace$:
	\begin{equation}
		a_n=\frac{1}{n^{3/2}}\left(C+\frac{c_1}{n^{1}}+...+\frac{c_{K}}{n^{K}}+ O(\frac{1}{n^{K+1}})\right).
		\end{equation}
	And $C=-\frac{h'(0)}{2\sqrt{\pi}}$. Similar formulas can be obtained for the other constants $c_l$, but we won't calculate them here.
\end{Thm}

\begin{proof}
	Fix $\alpha=1/2$. We repeat the previous proof, changing what needs to be changed: 
	Let $K\in\mathbb{N}\setminus\{0\}$ be an integer. Recall the notation for $\beta\in\mathbb{R}$, $\theta\in ]0,2\pi[$: $p_\beta(\theta):=(1-e^{i\theta})^\beta$. From Cauchy's formula, the coefficients $a_n$ satisfy:
	$$a_n=\frac{1}{2i\pi} \int_{\partial \mathbb{D}} \frac{g(z)}{z^{n+1}}dz.$$
	Parametrizing the disk boundary $\partial\mathbb{D}$ by the function $\theta\mapsto e^{i\theta}$ defined on $]0,2\pi[$, we obtain:
	\begin{equation*}
		a_n=\frac{1}{2\pi}\int_0^{2\pi}g(e^{i\theta})e^{-in\theta}d\theta=\frac{1}{2\pi}\int_0^{2\pi}h(p_\alpha(\theta))e^{-in\theta}d\theta.
	\end{equation*}
	Let $m\in\mathbb{N}$ be such that $ m\alpha-K>0$ and, from the Lemma \ref{Decomposition_Lemma} with $z_0=0$, let $\tilde{h}_m$ be the holomorphic function defined on a neighbourhood of $\overline{\mathbb{D}(1,1)^\alpha}$ such that for any complex number $z$ in this neighbourhood, we have:
	\begin{gather*}
		h(z)=\sum_{k=0}^{m-1} \frac{h^{(k)}(0)}{k!}  z^k + z^m\tilde{h}_m(z)\\
 		\text{and}\\
 		\tilde{h}_m(0)=\frac{h^{(m)}(0)}{m!}.
	\end{gather*}
	
	Then, as before we get
	\begin{equation}\label{formule_de_a_n_avancee}
		a_n=\sum_{k=0}^{m-1}  \frac{h^{(k)}(0)}{k!}\left(\frac{1}{2\pi}\int_0^{2\pi} p_\alpha(\theta)^k e^{-in\theta} d_\theta\right)
		+ \frac{1}{2\pi}\int_0^{2\pi} \tilde{h}_m(p_\alpha(\theta))p_\alpha(\theta)^m e^{-in\theta} d\theta.
	\end{equation}
	Note that for any integer $k\in\mathbb{N}$, $p_\alpha(\theta)^k=p_{\alpha k}(\theta)$ and that we already have estimates for the integrals $c_n(\alpha k)=\displaystyle\frac{1}{2\pi}\int_0^{2\pi} p_\alpha(\theta)^k e^{-in\theta} d\theta$ (Lemma \ref{Advanced_Root_Lemma}): For any integer $k\in\mathbb{N}$, we have:
	\begin{equation}\label{egalite_2}
		\alpha k\notin\mathbb{N},\quad c_n(\alpha k)=-\frac{\alpha k}{\Gamma(1-\alpha k)}\frac{1}{n^{\alpha k+1}}\left( 1 + \frac{e_1}{n} + \frac{e_2}{n^2} + ... + \frac{e_{K-1}}{n^{K-1}}+O \left(\frac{1}{n^{K}}\right)\right);
	\end{equation}
	\begin{equation}\label{egalite_3}
		\alpha k \in \mathbb{N},\quad c_n(\alpha k) \text{ is zero when } n>\alpha k.
	\end{equation}

Where the sequence of constants $(e_l)_{l\geq 1}$ in (\ref{egalite_2}) depends on $\alpha k$. Then
	$$ \frac{1}{2\pi}\int_0^{2\pi} \tilde{h}_m(p_\alpha(\theta))p_\alpha(\theta)^m e^{-in\theta} d_\theta = c_n\left((\tilde{h}_m\circ p_\alpha)p_{\alpha m}\right).$$
	The function $\varphi:\theta\mapsto\tilde{h}_m(p_\alpha(\theta))p_{\alpha m}(\theta)$ is $K$ times derivable and its $K$-th derivative is integrable on $]0,2\pi[$, when $\alpha m - K > -1$.
	Indeed, noting that for any $\beta\in\mathbb{R}$, $\forall\theta\in ]0,2\pi[,$ we have $p_\beta'(\theta)=\beta (-ie^{i\theta}) p_{\beta-1}(\theta)$, we obtain by deriving $K$ times the function $\varphi$ that $\varphi^{(K)}$ is a sum of functions of the form:
	\begin{equation}\label{eq-t26}
		C e^{ik\theta}\tilde{h}_m^{(l)}(p_\alpha(\theta))p_{\alpha(m+l)-l-l'}(\theta),
	\end{equation}
	with $k\leq K $ and $l+l'\leq K$. Functions of this forms are integrable over $]0,2\pi[$ as long as we have $\alpha(m+l)-l-l'>-1$. Since $l+l'\leq K$, and $\alpha\geq 0$, we have $\alpha(m+l)-l-l'\geq \alpha m - K> -1$, which is satisfied by choice of $m$.

Note also that for any $\beta > 0$, we have \begin{equation}
	\lim_{\theta\to 0} p_\beta(\theta)=\lim_{\theta\to 2\pi}p_\beta(\theta).
\end{equation} From the inequality $\alpha m - K>0$, for $l\leq 2$, the same analogous limit hold with the functions (\ref{eq-t23}). We deduce that 
$$\lim_{\theta\to 0} \varphi^{(k)}(\theta)=\lim_{\theta\to 2\pi}\varphi^{(k)}(\theta)=0,$$ for any $k\leq K$. And in particular $c_0(\varphi^{(k)})=0$ for any $k\leq K$.
	 
	 Then, iteratively applying the Lemma  \ref{Lemma_d_endadrement_Fourrier} to the $K$ first derivatives of $\varphi$, knowing that $c_0(\varphi^{(k)})=0$, for any $k\leq K$, we obtain for $n\geq 1$:
	\begin{equation}\label{formule_du_reste_de_a_n_avancee}
		\frac{1}{2\pi}\int_0^{2\pi} \tilde{h}_m(p_\alpha(\theta))p_\alpha(\theta)^m e^{-in\theta} d_\theta =O\left(\frac{1}{n^{K}}\right).
	\end{equation}
	Finally, by introducing the estimates (\ref{egalite_2}),(\ref{egalite_3}) and (\ref{formule_du_reste_de_a_n_avancee}) into the formula for $a_n$ (\ref{formule_de_a_n_avancee}) we obtain for all $n>\alpha K$:
	\begin{align*}
		a_n&= \sum_{\substack{k=1\\\alpha k \notin \mathbb{N}}}^{m-1} -\frac{h^{(k)}(0)}{k!}\frac{\alpha k}{\Gamma(1-\alpha k)}\frac{1}{n^{\alpha k+1}}
 		\left( 1 + \frac{e_1}{n} + \frac{e_2}{n^2} + ... + \frac{e_{K-1}}{n^{K-1}}+O \left(\frac{1}{n^{K}}\right) \right)\\
		&\, + O(\frac{1}{n^K}).
	\end{align*}
	Remembering that $\alpha=1/2$, we have that $\alpha k$ is an integer whenever $k$ is even. The result follows after simplification.
\end{proof}

\begin{Rem}
	Actually, if we replace the exponent $1/2$ by some real number $\alpha\in]0,1[$, in the previous tauberian theorems, that is to say, $g$ is of the form $g(z)=h((1-z)^{\alpha})$ then if $\sum_n a_n z^n$ denote the power series expansion of $g$ near $z=0$, we get with light modification of the proof, the existence of a unique sequence $(c_j)_j$ for $j\in\{ n\alpha + m: n,m\in\mathbb{N}\setminus\{0\}\}$ such that for any $K$ we have the asymptotic:
	\begin{equation}
		a_n=\sum_{\substack{j<K\\ j\,\notin\, \mathbb{N}}} \frac{c_{j}}{n^{j}} + O\left( \frac{1}{n^K}\right).
	\end{equation}
	And the first constant $c_{\alpha + 1}$ equals  $-\frac{h'(0)\alpha}{\Gamma (1-\alpha)}$.
\end{Rem}
\begin{Thm}    \label{Thm_Tauberien_pole_tout_ordre}

	Let $g$ be a continuous function on $\overline{\mathbb{D}}$ (with value in $\mathbb{C}\cup\{\infty\}$) whose restriction to $\mathbb{D}$ is holomorphic. Let $\displaystyle\sum_{n\in\mathbb{N}} b_n z^n$ be the power series expansion of $g$ in the neighbourhood of $0$.
 	Suppose there exists $h_p$ a meromorphic function defined on a neighbourhood of $\overline{\mathbb{D}(1,1)^{1/2}}\subset\mathbb{C}$ that only possesses a pole at $0$ with multiplicity $M\geq 1$, and suppose that $h_p$ satisfies
 	\begin{gather*}
 		\forall z\in\overline{\mathbb{D}},\, g(z)=h_p(\sqrt{1-z})
	\end{gather*}
	Then there exists a unique sequence of constants $(d_n)_{n\geq 1}$ and a constant $D$ such that for any positive integer $K\in\mathbb{N}\setminus\lbrace 0\rbrace$:
	\begin{equation}
		b_n=\frac{D}{n^{3/2-M}}\left(1+\frac{d_1}{n^{1}}+...+\frac{d_{K}}{n^{K}}+ O(\frac{1}{n^{K+1}})\right)
		\end{equation}
	And if $M=1$ then $D=-\dfrac{Res(h_p,0)}{\sqrt{\pi}}$. Similar formulas can be obtained for the other constants, but we won't calculate them here.
\end{Thm}

\begin{proof}
By assumption on $h_p$ there exists a holomorphic function $h$ in neighbourhood of $\overline{\mathbb{D}(1,1)^{1/2}}\subset\mathbb{C}$ and constants $D_M,..., D_1$ such that 
$$h_p(w)=\frac{D_M}{w^M}+\dots +\frac{D_1}{w}+h(w).$$

In particular we get
	\begin{gather*}
 		\forall z\in\overline{\mathbb{D}},\, g(z)=
 		\frac{D_M}{(\sqrt{1-z})^M}+\dots+ \frac{D_1}{\sqrt{1-z}}+h(\sqrt{1-z}).
	\end{gather*}
	
	Then applying the Lemma \ref{Advanced_Root_Lemma} to the quotients $\frac{D_j}{\sqrt{1-z}^j}$ for $1\leq j \leq M$ and using Theorem \ref{Thm_Tauberien_tout_ordre} on the function $z\mapsto h(\sqrt{1-z})$, we immediately obtain the desired asymptotic expansion. 
\end{proof}

		\section{Ramified coverings and local map singularities}
we will note $\sqrt{.}$ for the square root function: $\pi_{1/2}(.)$ defined in Definition \ref{def_power_functions}.

\begin{Lem}~

	Let $\mathcal{C}$ be a smooth analytic complex curve (i.e. a Riemann surface or a complex manifold of dimension 1). 
	Let $u:\mathbb{D}\to\mathcal{C}$ and $\lambda:\mathcal{C}\to\mathbb{C}$ be two holomorphic functions. Assume that
	\begin{enumerate}[label=\roman*)]
		\item $\lambda\circ u=id_{\mathbb{D}}$;
		\item $u$ extends by continuity into a function $\overline{\mathbb{D}}\to\mathcal{C}$;
		\item The first derivative of $\lambda$ at $p=u(1)$ is zero: $D_p\lambda=0$;
		\item The second derivative of $\lambda$ at $p$ is non-zero 
			\footnote{That is,\index{Second derivative} if $c:\mathcal{U}_p\to\mathbb{C}$ is a local map of $\mathcal{C}$ such that $c(p)=0$, then $(\lambda\circ c^{-1})'(0)=0$ and $(\lambda\circ c^{-1})''(0)\neq 0$. Since $D_p\lambda=0$, the relationship $D^2_p\lambda\neq 0$ is intrinsic.\\
			A \textit{local map}\index{Local map} in the neighbourhood of a point $p$, of a curve $\mathcal{C}$, is given by the data of an open neighbourhood $\mathcal{U}_p\subset\mathcal{C}$ of $p$, and a complex-valued function $c:\mathcal{U}_p\to\mathbb{C}$, bijective on its image and open (i.e. whose range over open sets is open). For the complex differential structure of the Riemann surface $\mathcal{C}$, these maps are biholomorphisms onto their image.}
			: $D^2_p \lambda\neq 0$.
	\end{enumerate}

	Under these conditions, there exists a holomorphic function $h$ defined on a neighbourhood of $z=0$, with values in $\mathcal{C}$, such that $h(0)=p$ and the derivative of $h$ at $0$ is non-zero: 
	$$h'(0)\in T_p\mathcal{C}\setminus\{0\},$$ 
	and there exists a neighbourhood $V_1$ of $1$ in $\mathbb{C}$ such that, for any complex number $z\in V_1\cap\mathbb{D}$, 
	$$u(z)=h(\sqrt{1-z})$$
\end{Lem}
\begin{proof}
	Let $c:\mathcal{U}_p\to\mathbb{C}$ be a local map in the neighbourhood of $p$, such that $c(p)=0$. To reduce our work to studying complex functions, let 
	$$\lambda_c:=\lambda\circ c^{-1} \text{ and }u_c:=c\circ u.$$ 
	They verify $\lambda_c(0)=\lambda(p)=\lambda(u(1))=1$. 
	The function $\lambda_c$ is holomorphic on a neighbourhood of $0\in\mathbb{C}$ and the function $u_c$ is holomorphic on $V\cap\mathbb{D}(0,1)$ where $V$ is a neighbourhood of $1\in\mathbb{C}$. By assumption, $\lambda_c'(0)=0$ and $\lambda_c^{''}(0)\neq 0$.
	Using the Lemma \ref{Decomposition_Lemma} for $z_0=0$ and $m=2$, there exists a holomorphic function $\tilde{h}_2$ defined in a neighbourhood of $0$ such that $\tilde{h}_2(0)=2\lambda_c^{''}(0)\neq 0$ and for any complex number $z$ in this neighbourhood,
	$$\lambda_c(z)=1-z^2 \tilde{h}_2(z).$$
	Then define $$h_c:z\mapsto z \sqrt{\tilde{h}_2(z)},$$
	which is well-defined and holomorphic in a neighbourhood $W\subset\mathbb{C}$ of $0$, so that $h_c(0)=0$. Its derivative at $0$ verifies $h_c'(0)=\sqrt{\tilde{h}_2(0)}\neq 0$. Thus $h_c$ is invertible in the neighbourhood of $0$, onto its range, with holomorphic inverse.
	If we reduce $W$, we can assume without loss of generality, that $h_c:W\to\mathbb{C}$ is a biholomorphism onto its range.
	
	Now, if $V_1$ denotes a neighbourhood of $1$ in $\mathbb{C}$ such that the image of $V_1\cap\mathbb{D}$ by $u_c$ satisfies: 
	$$u_c(V_1\cap\mathbb{D})\subseteq W,$$
	(Such a neighbourhood exists by continuity of $u_c=c\circ u$) then for any complex number $z$ in $V_1\cap\mathbb{D}$:
	$$ z=\lambda\circ u(z)=\lambda_c\circ u_c(z)= 1-(h_c(u_c(z)))^2.$$
	And so for any $z\in V_1\cap\mathbb{D}$, there exists a $\varepsilon\in\{-1,1\}$ such that:
	$$\varepsilon\sqrt{1-z}=(h_c\circ c)(u (z)).$$
 	Even if it means replacing $h_c$ by $-h_c$, and restricting $V_1$ in order to have $V_1\cap\mathbb{D}$ connected, we can assume that $\varepsilon=1$.

	Thus, for all $z\in V_1\cap\mathbb{D}$, we have
 	$$ u(z)=(h_c \circ c)^{-1}(\sqrt{1-z}).$$
	Finally, we set $h:=(h_c \circ c)^{-1}$, well-defined and holomorphic in a neighbourhood of $0$, to obtain the desired result, i.e. $u(z)= h(\sqrt{1-z})$ in a neighbourhood of $1$ and $h'(0)\neq 0$.
\end{proof}

\begin{Prop}\label{Prop_Composition_racine}~

	Under the same assumptions as the previous lemma, if we further assume that
	\begin{enumerate}[label=\roman*), start=5]
		\item $u$ extends holomorphically to the neighbourhood of any point of $\partial\mathbb{D}\setminus\{1\}$, the boundary of $\mathbb{D}$ minus $1$.
	\end{enumerate}
	Then there is a holomorphic map $h$ defined on a neighbourhood of the closure $\overline{\mathbb{D}(1,1)^{1/2}}$ in $\mathbb{C}$, with values in $\mathcal{C}$ such that $h'(0)\in T_p\mathcal{C}$ is non-zero
	and for any complex number $z\in\overline{\mathbb{D}}$, we have
	$$u(z)=h(\sqrt{1-z}).$$
\end{Prop}
\begin{proof}
	We begin by applying the previous lemma, obtaining a neighbourhood $V_1$ of $1$ in $\mathbb{C}$, a neighbourhood $W$ of $0\in\mathbb{C}$, such that the image of $V_1\cap\mathbb{D}$ by the function $z\mapsto\sqrt{1-z}$ is included in $W$; and a holomorphic map $h:W\to\mathcal{C}$ such that $h'(0)\neq 0$ and $\forall z\in V_1\cap\mathbb{D}$, $$u(z)=h(\sqrt{1-z}).$$
	Without loss of generality, we can assume that $V_1=\mathbb{D}(1,\varepsilon)$, where $\varepsilon>0$, and that $W$ is convex.

	Note that $z\mapsto\sqrt{1-z}$ is a biholomorphism of $\mathbb{D}(0,1)$ onto its range, and its inverse is given by $w\mapsto 1-w^2$. We extend $h$ by $h(w)= u (1-w^2)$ on a neighbourhood of the adherence of the set $\{\sqrt{1-z}: z \in \mathbb{D}(1,1)\setminus V_1\}$ (See figure \ref{Opens_figure}).
	In accordance with the monodromy theorem (See\cite{Rudin_FR_2009}, Theorem 16.15), this extension is possible insofar as the union of the sets $W$ and $\overline{\mathbb{D}(1,1)\setminus V_1}$ forms a simply connected set. And note that the previous map $w\mapsto u(1-w^2)$ is well-defined on some neighbourhood of $\overline{\mathbb{D}(1,1)\setminus V_1}$ thanks to the assumption \textit{v)} of the proposition.
\end{proof}
\begin{figure}[h]
    \centering
    \includegraphics[scale=0.6]{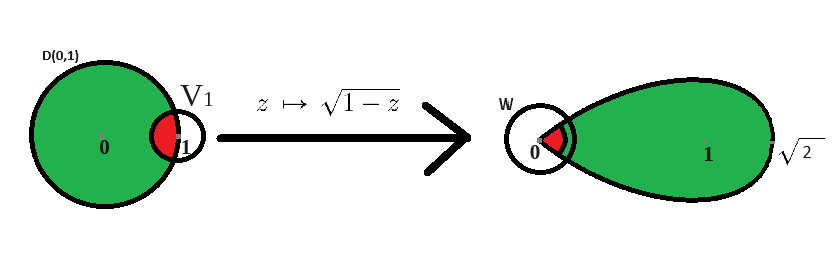}
    \caption{}
    \label{Opens_figure}
\end{figure}
\section{Asymptotics for ramified coverings of degree two}
\begin{Thm}\label{Thm_Black_Box_aperiodic}
Let $\mathcal{C}$ be a smooth analytic complex curve (i.e. a Riemann surface or a complex manifold of dimension 1) and let $R>0$ be a positive real number. 
	Let $u:\mathbb{D}(0,R)\to\mathcal{C}$ and $\lambda:\mathcal{C}\to\mathbb{C}$ be two holomorphic maps and let $g$ be a complex valued function defined and holomorphic in a neighbourhood of the closure $\overline{\{ u(z): z\in\mathbb{D}(0,R)\}}\subset\mathcal{C}$. Denote by $\sum_{n\in\mathbb{N}} a_n z^n$ the power series expansion of the complex valued function $g\circ u$. Assume that
	\begin{enumerate}[label=\roman*)]
		\item $\lambda\circ u=id_{\mathbb{D}(0,R)}$;
		\item $u$ extends by continuity into a map $\overline{\mathbb{D}}(0,R)\to\mathcal{C}$;
		\item The first derivative of $\lambda$ at $p=u(R)$ is zero: $D_p\lambda=0$;
		\item The second derivative of $\lambda$ at $p$ is non-zero over $T_{p}\mathcal{C}$: $D^2_p \lambda\neq 0$;
			\item $u$ extends holomorphically to the neighbourhood of any point on the boundary of $\mathbb{D}(0,R)$ minus $R$: $\partial\mathbb{D}(0,R)\setminus\{R\}$.
			\item The derivative of $g$ at $p$ is non-zero: $D_p g\neq0$.
	\end{enumerate}
	 Then there exists a unique non-zero constant $C$, and a unique sequence of constants $(c_l)_{l\geq 1}$, such that for any positive integer $K\in\mathbb{N}\setminus\lbrace 0\rbrace$, we have the asymptotic expansion
	\begin{equation}\label{eq-t14}
		a_n=CR^{-n}\frac{1}{n^{3/2}}\left(1+ \frac{c_1}{n^{1}}+ ... +\frac{c_{K}}{n^{K}}+ O\left(\frac{1}{n^{K+1}}\right) \right).
	\end{equation}
\end{Thm}
\begin{proof}
	We prove the case when $R=1$ for the general case is just a Corollary of this one.
	
	Under the above assumptions, we can apply the Proposition \ref{Prop_Composition_racine} obtaining a holomorphic map $h$ defined on a neighbourhood of the closure $\overline{\mathbb{D}(1,1)^{1/2}}\subset\mathbb{C}$, with values in $\mathcal{C}$ such that: 
	\begin{enumerate}[label=\arabic*)]
		\item $h'(0)\in T_p\mathcal{C}$ is non-zero;
		\item For any complex number $z\in\overline{\mathbb{D}}$,  we have $u(z)=h(\sqrt{1-z}).$
	\end{enumerate}
	
	Then for any complex number $z\in\mathbb{D}$, we have  $g\circ u(z)= g\circ h (\sqrt{1-z})$. Applying the theorem \ref{Thm_Tauberien_tout_ordre}, knowing that $(g\circ h)'(0)=D_p g(h'(0))\neq 0$, we obtains the desired result, that is to say:
	
	There exists a unique sequence of constants $(c_l)_{l\geq 1}$ and a unique constant $C\neq 0$ such that, for any positive integer $K\in\mathbb{N}\setminus\lbrace 0\rbrace$, the asymptotic (\ref{eq-t14}) holds. Besides $C=-\frac{(g\circ h)'(0)}{2\sqrt{\pi}}\neq 0$.
\end{proof}

\begin{Cor}\label{Cor_Black_Box_Pole_aperiodic}
Let $\mathcal{C}$ be a smooth analytic complex curve (i.e. a Riemann surface or a complex manifold of dimension 1). 
	Let $u:\mathbb{D}(0,R)\to\mathcal{C}$ and $\lambda:\mathcal{C}\to\mathbb{C}$ be two holomorphic maps and let $g$ be a meromorphic function in a neighbourhood of the closure $\overline{\{ u(z): z\in\mathbb{D}(0,R)\}}\subset\mathcal{C}$. Denote by $\sum_{n\in\mathbb{N}} b_n z^n$ the power series expansion of the complex valued function $g\circ u$. Assume that
	\begin{enumerate}[label=\roman*)]
		\item $\lambda\circ u=id_{\mathbb{D}(0,R)}$;
		\item $u$ extends by continuity into a continuous map $\overline{\mathbb{D}}(0,R)\to\mathcal{C}$;
		\item The first derivative of $\lambda$ at $p=u(R)$ is zero: $D_p\lambda=0$;
		\item The second derivative of $\lambda$ at $p$ is non-zero: $D^2_p \lambda\neq 0$;
			\item $u$ extends holomorphically to the neighbourhood of any point on the boundary of $\mathbb{D}$ minus $R$: $\partial\mathbb{D}(0,R)\setminus\{R\}$.
			\item $g$ possesses one, and only one pole at $u(R)$ in $\overline{\{ u(z): z\in\mathbb{D}(0,R)\}}$ of degree $M\geq 1$
	\end{enumerate}
	 Then there exists a unique non-zero constant $D$ and a unique sequence of constants $(d_n)_{n\geq 1}$ such that for any positive integer $K\in\mathbb{N}\setminus\lbrace 0\rbrace$, we have the asymptotic expansion
	\begin{equation}\label{eq:e5}
		b_n=\frac{DR^{-n}}{n^{3/2-M}}\left(1+ \frac{d_1}{n^{1}}+ ... +\frac{d_{K}}{n^{K}}+ O\left(\frac{1}{n^{K+1}}\right) \right).
	\end{equation}
	Besides, if $M=1$ then $D= -\frac{Res(g\circ h,0)}{\sqrt{\pi}}$.
\end{Cor}
\begin{proof}
We only prove the case when $R=1$, for the general case $R>0$ is a Corollary of this one.

Under these assumptions, we can apply the Proposition \ref{Prop_Composition_racine} obtaining a holomorphic map $h$ defined on a neighbourhood of the closure $\overline{\mathbb{D}(1,1)^{1/2}}\subset\mathbb{C}$, with values in $\mathcal{C}$ such that: 
	\begin{enumerate}[label=\arabic*)]
		\item $h'(0)\in T_p\mathcal{C}$ is non-zero;
		\item For any complex number $z\in\overline{\mathbb{D}}$, we have $u(z)=h(\sqrt{1-z})$.
	\end{enumerate}
	
Since $h'(0)\neq 0$, there exists a neighbourhood $U_0$ of $0$ in $\mathbb{C}$, and a neighbourhood $V$ of $h(0)=u(1)$ in $\mathcal{C}$ such that $h_{|U_0}:U_0\to V$ is bijective, in particular $h^{-1}(\{u(1)\})\cap U_0=\{0\}$. Besides, for any complex number $w$ in the compact set $\overline{\mathbb{D}(1,1)}\setminus U_0$, we have $h(w)\neq u(1)$, thus, even if it means reducing the domain $\mathcal{D}_h\supseteq \overline{\mathbb{D}(1,1)}$ of definition of $h$, we can suppose that $h^{-1}(\{u(1)\})=\{0\}$.

With this last assumption, the meromorphic function $w\mapsto g\circ h (w)$ over $\mathcal{D}_h$ possesses only a pole at $z=0$, that has degree $M$. Thus applying the Theorem \ref{Thm_Tauberien_pole_tout_ordre} to $g\circ h$, we obtain the desired asymptotic (\ref{eq:e5}).
\end{proof}

	\section{Extension to \texorpdfstring{$\mathbb{Z}/d\mathbb{Z}$}{-Z/dZ}-equivariant maps}\label{Subsection_Extension_to_ZdZ_action}
			For dealing with the case of aperiodic random walks, in this part we are interested in maps $u$, defined on disks $\mathbb{D}(0,R)\subset \mathbb{C}$, with values in a Riemann surface $\mathcal{C}$, satisfying, for some automorphism $A\in Aut(\mathcal{C})$ of order $d$, the formula 
			$$\forall z\in\mathbb{D}(0,R),\, u(e^{2i\pi/d} z)=A u(z).$$

We prove here the following Corollary of theorem \ref{Thm_Black_Box_aperiodic}:

\begin{Cor} \label{Cor_Black_box_general}
Let $R>0$ be a real number, $d$ be a positive integer, $\mathcal{C}$ be a smooth analytic complex curve (i.e. a Riemann surface or a complex manifold of dimension 1) and $A\in Aut(\mathcal{C})$ of order $d$ (i.e $A,...,A^{d-1}\neq Id_\mathcal{C}$, $A^d=Id_\mathcal{C}$). Let $u:\mathbb{D}(0,R)\to\mathcal{C}$ and $\lambda:\mathcal{C}\to\mathbb{C}$ be two holomorphic maps and let $g$ be a complex valued function defined and holomorphic in a neighbourhood of the closure $\overline{\{ u(z): z\in\mathbb{D}(0,R)\}}\subset\mathcal{C}$. Denote by $\sum_{n\in\mathbb{N}} a_n z^n$ the power series expansion of the complex valued function $g\circ u$ near $z=0$. Assume that
	\begin{enumerate}[label=\roman*)]
		\item $\lambda\circ u=id_{\mathbb{D}(0,R)}$;
		\item $u$ extends by continuity into a continuous map $\overline{\mathbb{D}(0,R)}\to\mathcal{C}$;
		\item The first derivative of $\lambda$ at $p=u(R)$ is zero over $T_p \mathcal{C}$: $D_p\lambda=0$;
		\item The second derivative of $\lambda$ at $p$ is non-zero 
			: $D^2_p \lambda\neq 0$;
		\item $u$ extends holomorphically to the neighbourhood of any point on the boundary of $\mathbb{D}(0,R)$ minus the $d$-th roots of $R^d$: $$\partial\mathbb{D}(0,R)\setminus\left\lbrace R(e^{2i\pi/d})^k: k\in \{0,...,d-1\}\right\rbrace;$$
		\item The derivative of $g$ at $p$ is non-zero over $T_p\mathcal{C}$: $D_p g\neq 0$.
		\item $Fix(A^1)=...=Fix(A^{d-1})=\{u(0)\}$ and for any $ z\in\mathbb{D}(0,R)$, we have $$u(e^{2i\pi/d}z)=Au(z);$$
		\item There is an integer $r\in\{0,...,d-1\}$ such that for any $q\in \overline{\{ u(z): z\in\mathbb{D}(0,R)\}}$, $$g(Aq)=e^{2i\pi r/d}g(q).$$
	\end{enumerate}
	
	 Then for any non-negative integer $n$ such that $n\not\equiv r \, [d]$, we have $a_n=0$ and there exists a non-zero constant $C$ and a unique sequence of constants $(c_l)_{l\geq 1}$ such that for any positive integer $K\in\mathbb{N}\setminus\lbrace 0\rbrace$ we have the asymptotic expansion:
	\begin{equation}
		a_{dn+r}=CR^{-dn}\frac{1}{n^{3/2}}\left(1 + \frac{c_1}{n^{1}}+...+\frac{c_K}{n^{K}}+O\left(\frac{1}{n^{K+1}}\right)\right).
	\end{equation}

\end{Cor}
\begin{Rem}
        Actually, we only need the curve $\mathcal{C}$ to be smooth in a neighbourhood of the closure $\overline{\{ u(z): z\in\mathbb{D}(0,R)\}}\subset\mathcal{C}$.
\end{Rem}
\begin{Rem}\label{Rem_Cor_Black_box_general}
    Through the computation it can be shown that the constant $C$ in the above asymptotic expansion equals $C=-\frac{D_p g (h'(0))}{2\sqrt{\pi}}$, where $h$ is some holomorphic map in a neighbourhood of $\overline{\sqrt{\mathbb{D}(1,1)}}$ with non-zero derivative in $T_p\mathcal{C}$, that does not depend on $g$.
\end{Rem}

We prove this Corollary in several steps. We start by reducing the problem to the case where the function $g$ is $A$-invariant, i.e such that $r=0$ in assumption $viii)$. Secondly we briefly introduce some results on lifting and lowering maps through the canonical surjection $\Pi:\mathcal{C}\to\mathcal{C}_A$, where $\mathcal{C}_A$ is the set of $A$-orbits of elements in $\mathcal{C}$ (see notation \ref{Notation-K1} below), in order to apply the Proposition \ref{Prop_Composition_racine} with the Riemann surface $\mathcal{C}_A$ and the maps $\lambda_A, u_A$ which are lowering of the maps $\lambda$ and $u$, respectively, through $\Pi$. Lastly, we use theorem \ref{Thm_Black_Box_aperiodic} to conclude.

\begin{Notation}\label{Notation-K1}
	All along this section, we will work under the hypotheses of the Corollary \ref{Cor_Black_box_general}.
	Since the action of $A$ generates a finite group of automorphisms of $\mathcal{C}$, the set of its orbits $\mathcal{C}_A$ can be given a Riemann surface structure (see Proposition \ref{Prop_R_C_d} below) making $\Pi$ a submersion. Actually the map $\Pi$ is ramified at the point $u(0)$ fixed by $A$.
	
	We will denote by $o$ the point $u(0)$ in $\mathcal{C}$ and we will write $o_A=\Pi(o)$ in $\mathcal{C}_A$, for its image by $\Pi$ in $\mathcal{C}_A$. 
\end{Notation}

\subsection{Reduction to the \texorpdfstring{$A$}{-A}-invariant case}\label{subsection_reduction_A_inv}

We work under the assumptions of the Corollary \ref{Cor_Black_box_general}. By hypotheses, the map $z\mapsto g\circ u(z)$ satisfies for some $r\in \{0,...,d-1\}$:
\begin{equation}
	\forall z\in\mathbb{D}(0,R), g\circ u(e^{2i\pi/d} z)= (e^{2i\pi/d})^r g\circ u(z)
\end{equation}
\begin{Lem}\label{Multiplicity_of_the_zeros}
		Let $f:\mathbb{D}\to\mathbb{C}$ be a holomorphic function, and let $0\leq r < d$ be an integer. Denote by $\sum_{n\geq 0} f_n z^n$ its power series expansion in the neighbourhood of $z=0$. If for any complex number $z\in\mathbb{D}$, $$f(e^{2i\pi/d} z)=(e^{2i\pi/d})^r f(z),$$
		then for any non-negative integer $n\geq 0$, we have that $f_n=0$ when $n\neq r\, [d]$.
	\end{Lem}
	\begin{proof}
    	Let $\zeta_d=e^{2i\pi/d}$. For any complex number $z\in\mathbb{D}$, we have
	$$f(\zeta_d z ) =\sum_{n\geq 0} f_n \zeta_d^n z^n = \sum_{n\geq 0} f_n \zeta_d^{r} z^n=\zeta_d^r f(z).$$
	Thus for any non-negative integer $n\geq 0$, we have
	$$ f_n\zeta_d^n = f_n \zeta_d^r.$$
	thus for any non-negative integer $n$, if $n\neq r \, [d]$, necessarily $f_n=0$.
	\end{proof}

We apply the previous Lemma \ref{Multiplicity_of_the_zeros}, to $f=g\circ u$, obtaining that $z=0$ is a zero of $g\circ u$ of multiplicity at least $r$. Thus the function $h=\frac{g}{\lambda^r}$ is a meromorphic function on $\mathcal{C}$ that is holomorphic on the closure $\overline{\{u(z): z \in \mathbb{D}(0,R)\}}\subset \mathcal{C}$ (for $ h\circ u(z)= \frac{g\circ u (z)}{z^r}$ and $u$ parametrizes $\mathcal{C}$ in a neighbourhood of $o$). 

\begin{Rem}\label{Rem_Reduction}
	This map $ h=\frac{g}{\lambda^r}$ is $A$-invariant and the power series expansion of $h\circ u$ in the neighbourhood of $z=0$ is $\sum_n a_{n-r} z^n $, where $\sum_n a_n z^n$ is the power series expansion of $g\circ u$.
\end{Rem}

From this last remark, we deduce that we just need to prove the Corollary for $g$ being an $A$-invariant map, that is to say when $r=0$ in assumption $viii)$ of the corollary.

\subsection{Lifting and Lowering maps through a ramified covering}
It is somehow well known that the following sujection $\Pi$ between Riemann surfaces is a ramified covering of degree $d$ at $o_A \in\mathcal{C}_A$:
\begin{Prop}[see for example \cite{Farkas-Kra_1992}\textit{ III.7.7 and III.7.8}]\label{Prop_R_C_d}~
	
		There exists local coordinates of $\mathcal{C}$ in a neighbourhood of $o$ and local coordinates of $\mathcal{C}_A$ in a neighbourhood of $o_A$, in which the submersion between Riemann surfaces: 
		$$\Pi:\mathcal{C}\to\mathcal{C}_A$$
is $z\mapsto z^d$.

	We say that $\Pi$ is a \textit{ramified covering} with a branching point of degree $d$ at $o_A=\Pi(o)$.
	\end{Prop}
	
We are interested in necessary and sufficient condition to lift (resp. lower) maps through $\Pi$, i.e given a map $\phi:\mathcal{C}\to\mathbb{D}$ and a map $\mu:\mathbb{D}\to\mathcal{C}$, does there exist maps $\phi_A: \mathcal{C}_A\to\mathbb{D}$ and $\mu_A:\mathbb{D}\to\mathcal{C}_A$ such that
\begin{equation}\label{Relation_LiftLower}
	\pi_d\circ\phi=\phi_A\circ \Pi \quad \quad \text{and} \quad \quad \Pi\circ \mu = \mu_A\circ\pi_d,
\end{equation}
where $\pi_d:z\mapsto z^d$. (resp. given $\phi_A$ and $\mu_A$ does there exist $\phi$ and $\mu$ such that (\ref{Relation_LiftLower}) holds).
The answer is given in the propositions \ref{Prop_Lowering} below:

\begin{Prop}[Lowering and Lifting maps]
\label{Prop_Lowering}
\label{Prop_Lifting}~

	\og Lowering map\fg : Let $\phi:\mathcal{C}\to\mathbb{D}$, respectively $\mu:\mathbb{D}\to\mathcal{C}$ be holomorphic maps that preserves the orbits, i.e such that
	$$\exists l\in\mathbb{N},\, \forall p \in \mathcal{C},\, \phi(A p)= (e^{2i\pi/d})^l \phi(p),$$
respectively such that
	$$\exists l\in\mathbb{N},\, \forall z \in \mathbb{D},\, \mu(e^{2i\pi/d}z)=A^l \mu(z).$$
Then there exists lowering of these maps, i.e there exists holomorphic maps $\phi_A:\mathcal{C}_A\to \mathbb{D}$, respectively $\mu_A:\mathbb{D}\to\mathcal{C}_A$ such that the following diagrams commute:
\begin{equation} \label{commutating_diagram_1}
\begin{tikzcd}
\mathcal{C} \arrow[r, "\Pi", two heads] \arrow[d, "\phi"'] & \mathcal{C}_A \arrow[d, "\phi_A"] \\
\mathbb{D} \arrow[r, "\pi_d", two heads]               & \mathbb{D}                                                                  
\end{tikzcd}
	\quad \text{, respectively  }\quad
	\begin{tikzcd}
		\mathcal{C} \arrow[r, "\Pi", two heads]                             		& \mathcal{C}_A \\
		\mathbb{D} \arrow[r, "\pi_d", two heads] \arrow[u, "\mu"] 
		& \mathbb{D}; \arrow[u, "\mu_A"'] 
	\end{tikzcd}
\end{equation}

\og Lifting map\fg : Reciprocally let $\phi_A:\mathcal{C}_A\to\mathbb{D}$, respectively $\mu_A:\mathbb{D}\to\mathcal{C}_A$ be holomorphic maps, that are not constant equal to $0$, resp $o_A$.
	
	Let $p\in\mathcal{C}$, $\tau\in\mathbb{D}\setminus\{0\}$ such that $\phi_A(\Pi(p))=\pi_d(\tau)$, then we have the equivalence between the two assertions: 
	\begin{enumerate}[label=\roman*)]
		\item There exists a unique holomorphic map $\phi:\mathcal{C}\to \mathbb{D}$ such that 
	$$ \pi_d\circ \phi= \phi_A\circ \Pi \quad \text{and}\quad \phi(p)=\tau;$$
		\item Every zero of $\phi_A\circ \Pi$ has multiplicity a multiple of $d$. 
	\end{enumerate}
	Respectively, let $\tau\in\mathbb{D}$ and $p\in\mathcal{C}\setminus \{o\}$ such that $\mu_A(\pi_d(\tau))=\Pi(p)$ then we have the equivalence between the two assertions: 
	\begin{enumerate}[label=\roman*)]
			\item There exists a unique holomorphic map $\mu:\mathbb{D}\to \mathcal{C}$ such that 
	$$ \Pi\circ \mu= \mu_A\circ \pi_d \quad \text{and}\quad \mu(\tau)=p;$$
		\item Every zero of $\mu_A\circ \pi_d$ have multiplicity a multiple of $d$.
	\end{enumerate}	
\end{Prop}
\begin{proof}
\og Lowering map\fg : We prove the result for $\phi$, the proof for $\mu$ being analogous.
$\Pi$ admits locals section on $\mathcal{C}_A \setminus \{o_A\}$. More precisely, for any point $p\in \mathcal{C}\setminus\{o\}$, there exists an open neighbourhood $U_{\Pi(p)}\subset \mathcal{C}_A$ of $\Pi(p)$, and a holomorphic map $\sigma:U_{\Pi(p)}\to \mathcal{C}$ such that 
	$$\Pi\circ\sigma = Id_{U_{\Pi(p)}} \text{   and    } \sigma(\Pi(p))=p .$$

	We then define naturally $\phi_A$ to obtain the commutative diagram (\ref{commutating_diagram_1}):
	$$\forall p_A\in\mathcal{C}_A,\, \phi_A(\omega):=\pi_d\circ \phi\circ \sigma (\omega),$$
	where $\sigma$ is a local section of $\Pi$. By $A$-invariance of $\pi_d\circ \phi $, $\phi_A$ is well-defined - that is to say $\phi_A$ does not depend on the choice of the section $\sigma$ -. $\phi_A$ is also holomorphic on $\mathcal{C}_A\setminus \{o_A\}$ and bounded in a pointed neighbourhood of $o_A$. In particular, $o_A$ is a removable singularity of $\phi_A$. It follows that $\phi_A$ extends to a holomorphic map on $\mathcal{C}_A$ that satisfies $$\pi_d\circ \phi = \phi_A\circ \Pi.$$
	We still denote by $\phi_A$ this extension. That's the map $\phi_A$ we desired.
	\vspace{\baselineskip}
	
\og Lifting map\fg : Remark that $i)$ always implies $ii)$, thus we just need to show that $ii)$ implies $i)$. We prove the result for $\mu_A$, the proof for $\phi_A$ being analogous. 
	
	Note that since $\mu_A$ is a holomorphic map that is not constant equal to $o_A$, the set $Z(\mu_A)$ of points where $\mu_A$ equals $o_A$, is discrete in $\mathbb{D}$.

The existence and uniqueness of a map $\mu$, lifting $\mu_A$, comes from the lifting of $\mu_A\circ\pi_d$ on $\mathbb{D}$ minus $\pi_d^{-1}(Z(\mu_A))$, and the prescribed value of $\mu$ at the point $\tau$ (see diagram below and Theorem \ref{Lifting_Thm} below). Remark that since $\pi_d$ is proper, $\pi_d^{-1}(Z(\mu_A))$ is a discrete in $\mathbb{D}$.
$$
\begin{tikzcd}
                                                                                        &  & \mathcal{C}\setminus\{o\} \arrow[dd, "\Pi", two heads]                                   \\
                                                                                        &  &                                                                           \\
\mathbb{D}\setminus \pi_d^{-1}(Z(\mu_A)) \arrow[rr, "\mu_A\circ \pi_d"] \arrow[rruu, "\mu"] &  & {\mathcal{C}_A\setminus\{o_A\}}
\end{tikzcd}
$$

Insofar as $\mathbb{D}\setminus \pi_d^{-1}(Z(g))$ is connected and since, by $ii)$, at every points $\omega$ in the set $\pi_d^{-1}(Z(\mu_A))$, the $\mathcal{C}_A$-valued map $\mu_A\circ\pi_d$ is, in local coordinates, of the form $z\mapsto h_\omega(z)^d$, where $h_\omega$ is a holomorphic function and is zero at $\omega$, the previous lifting is possible according to the theorem \ref{Lifting_Thm}.
 We then extend holomorphically, and in a unique way, $g$ to $\mathbb{D}$ by using the Riemann extension theorem for isolated removable singularities. 
	\end{proof}	
\begin{Rem}
In this proof we used the existence of local section for quotient map induced by a properly discontinuous action (See \cite{Lee_2013} Theorem 21.10), the Riemann extension theorem for isolated removable singularities (See \cite{Rudin_FR_2009} Theorem 10.20) and the Lifting Theorem:

\begin{Thm}[\href{https://analysis-situs.math.cnrs.fr/Relevement-des-applications.html}{Lifting Theorem} (See \cite{Analysis-situs-LiftingTheorem})] \label{Lifting_Thm}
	Let $\Pi:(Y,y_0)\to(X,x_0)$ be a covering on pointed space. Let $(Z,z_0)$ be a connected, locally path connected, pointed space and let $\bar{f}:(Z,z_0) \to (X,x_0)$ an map. Then $\bar{f}$ admits a lifting $f:(Z,z_0)\to(Y,y_0)$ if, and only if,
	$$f_*\left[ \pi_1(Z, z_0)\right] \subset p_*\left[\pi_1(Y, y_0)\right].$$
	\end{Thm}
\end{Rem}

\subsection{Lowering of \texorpdfstring{$\lambda$}{-lambda} and \texorpdfstring{$u$}{-u} through \texorpdfstring{$\Pi$}{-pi}}
	Using Proposition \ref{Prop_Lowering} on $u:\mathbb{D}(0,R)\to\mathcal{C}$ and $\lambda:\mathcal{C}\to\mathbb{D}(0,R)$ we obtains maps $u_A:\mathbb{D}(0,R^d)\to\mathcal{C}_A$ and $\lambda_A:\mathcal{C}_A\to\mathbb{D}(0,R^d)$ such that 
	$$ \pi_d\circ\lambda= \lambda_A\circ \Pi, \text{ and }\Pi\circ u = u_A\circ \pi_d.$$
	Remark that by assumption $ii)$ of the Corollary \ref{Cor_Black_box_general}, $u$ extends by continuity on $\overline{\mathbb{D}(0,R)}$, then so does $u_A$ on $\overline{\mathbb{D}(0,R^d)}$.
	
We claim that the maps $\omega\mapsto u_A(R^d \omega)$ and $\omega\mapsto \frac{1}{R^d}\lambda_A(\omega)$ satisfies the assumptions of Proposition \ref{Prop_Composition_racine}, and so there exists a map $h_A$, holomorphic in a neighbourhood of the closure $\overline{\mathbb{D}(1,1)^{1/2}}$ such that $h_A'(0)\in T_{u_A(R^d)}\mathcal{C}_A$ is non-zero and for any $\omega\in\mathbb{D}(0,R^d)$:
\begin{equation}\label{formula_t2}
	u_A(\omega)=h_A\left(\sqrt{1-\frac{\omega}{R^d}}\right).
\end{equation}

\begin{Clm}\label{Claim_1}
The quotient map $u_A$ extends holomorphically on the neighbourhood of any point on the boundary of $\mathbb{D}(0,R^d)$ minus $R^d$ and for any complex number $\omega\in\mathbb{D}(0,R^d)$,
		$$ 	\lambda_A\circ u_A(\omega)=\omega. $$	
\end{Clm}
\begin{proof}[Proof of claim \ref{Claim_1}]
	By assumptions $v)$ of Corollary \ref{Cor_Black_box_general}, $u$ extends holomorphically on the neighbourhood of any point on the boundary of $\mathbb{D}(0,R)$ minus $\{(e^{2i\pi/d})^k : k\in \{1,...,d\}\}$. Using local sections $s$ of $\pi_d$ and since $u_A=\Pi\circ u \circ s$ we obtains the desired result.
	
	For the second result of the lemma, note that away from zero, we locally have:
	$$\lambda_A\circ u_A=\lambda_A\circ \Pi\circ \sigma \circ u_A = \pi_d\circ\lambda\circ u \circ s = \pi_d\circ s = Id.$$
	Where we used the assumption $i)$ of the Corollary \ref{Cor_Black_box_general}. And obviously $\lambda_A\circ u_A(0)=\lambda_A(o_A)=0$.
	\end{proof}
\begin{Clm}\label{Claim_2}
	The first derivative of $\lambda_A$ at $u_A(R^d)$ is zero on the tangent space $T_{u_A(R^d)}\mathcal{C}_A$ of $\mathcal{C}_A$ at $u_A(R^d)$, 
		$$T_{u_A(R^d)}\lambda_A\equiv 0;$$ 
		
		And its second derivative at $u_A(R^d)$ is non-zero.
\end{Clm}
	\begin{proof}[Proof of claim \ref{Claim_2}]
	The quotient map $\Pi:\mathcal{C}\to\mathcal{C}_A$ is a local diffeomorphism everywhere but at $o=u(0)$, the claim thus follows from assumption $iii)$ and $iv)$ of the Corollary \ref{Cor_Black_box_general}.
%
	\end{proof}
	
In accordance to the previous lemmas, the Proposition \ref{Prop_Composition_racine} applies. Thus, there exists a holomorphic map $h_A$ in the neighbourhood of $\overline{\mathbb{D}(R,R)^{1/2}}$ with values in $\mathcal{C}_A$ such that formula (\ref{formula_t2}) holds.

\subsection{Proof of Corollary \ref{Cor_Black_box_general}}
Since $g$ is $A$-invariant, it factors through $\Pi$. Let $g_A:\mathcal{C}_A\to\mathbb{C}$ be its factors map i.e be the map such that $g_A\circ \Pi = g$.

We thus have for any complex number $z\in\mathbb{D}(0,R)$, $$g\circ u (z)= g_A\circ\Pi\circ u (z) = g_A\circ u_A (z^d).$$
From formula (\ref{formula_t2}) we get that for any complex number $z\in\mathbb{D}(0,R)$,
\begin{equation}\label{formulat_t3}
	(g\circ u)(z)= (g_A\circ h_A)\left(\sqrt{1-\left(\frac{z}{R}\right)^d}\right).
\end{equation}
Besides the derivative of $(g_A\circ h_A)$ at $0\in\mathbb{C}$ is 
\begin{equation}\label{formula_t4}
(g_A\circ h_A)'(0)=(D_{p_A}g_A)(h_A'(0)),
\end{equation} 
where $p_A=h_A(0)=u_A(R^d)$. Using a local section $\sigma$ of $\Pi$ in the neighbourhood of $p_A$, such that $\sigma(p_A)=p=u(R)$, we get $g_A=g\circ\sigma$ and its differential at $p_A$ is: 
$$D_{p_A}g_A=D_{p}g\circ D_{p_A}\sigma,$$ 
that is non-zero on $T_{p_A}\mathcal{C}$ by assumption $vi)$ of the Corollary \ref{Cor_Black_box_general}. Since $h_A'(0)\neq 0$, so is (\ref{formula_t4}).

It follows from theorem \ref{Thm_Tauberien_tout_ordre} that if $\sum_n b_n \omega^n$ is the power series expansion of $ \omega\mapsto (g_A\circ h_A)(\sqrt{1-\omega})$ in the neighbourhood of $\omega=0$, then we obtain a unique sequence of constants $(c_l)_{l\geq 1}$ and a constant $C$ such that for any positive integer $K\in\mathbb{N}\setminus\lbrace 0\rbrace$:
	\begin{equation}
		b_n=\frac{C}{n^{3/2}}\left(1+\frac{c_1}{n^{1}}+...+\frac{c_{K}}{n^{K}}+ O\left(\frac{1}{n^{K+1}}\right)\right)
		\end{equation}
	And $C=-\frac{(g_A\circ h_A)'(0)}{2\sqrt{\pi}}\neq 0$. 
	
	From this we deduce that the coefficients of the power series expansion $\sum_n a_n z^n$ of $g\circ u$ in the neighbourhood of $z=0$, satisfy for any positive integer $K\in\mathbb{N}\setminus\lbrace 0\rbrace$, the asymptotic expansion:
	\begin{equation}
		a_{dn}=CR^{-dn}\frac{1}{n^{3/2}}\left(1+\frac{c_1}{n^{1}}+...+\frac{c_{K}}{n^{K}}+ O\left(\frac{1}{n^{K+1}}\right)\right)
		\end{equation}
		and for any $n$ such that $n\not\equiv 0 \, [d]$, we have $a_n=0$. This is what we wanted when $g$ is $A$-invariant. For the general case the result follows from the previous paragraph and the remark \ref{Rem_Reduction}. 
		The Corollary \ref{Cor_Black_box_general} is proven.		
\section{Pole type singularity extension}
Now instead of considering a holomorphic function $g$ in a neighbourhood of $\overline{\{u(z): z\in\mathbb{D}(0,R)\}}$ such that for any $q$, 
\begin{equation}\label{covar_prop}
	g(Aq)=e^{2i\pi/d}g(q),
\end{equation} we will consider a \underline{meromorphic} function $g$, with the same equivariance property (\ref{covar_prop}), that admits poles at the ramification points $\{u(R e^{2i\pi k/d}) : k\in\{0,1,...,d-1\}\}$. We also present below a simpler case when $g$ admits poles in the $A$-orbit $\{u(R' e^{2i\pi k/d}): k\in\{0,1,...,d-1\}\}$, of $u(R')$, for some $R'\in]0,R[$.
\begin{Cor} \label{Cor_Black_box_pole}
Let $R>0$ be a real number, $d$ be a positive integer, $\mathcal{C}$ be a smooth analytic complex curve (i.e. a Riemann surface or a complex manifold of dimension 1) and $A\in Aut(\mathcal{C})$ of order $d$ (i.e $A,...,A^{d-1}\neq Id_\mathcal{C}$, $A^d=Id_\mathcal{C}$). Let $u:\mathbb{D}(0,R)\to\mathcal{C}$ and $\lambda:\mathcal{C}\to\mathbb{C}$ be two holomorphic maps and let $g$ be a meromorphic function in a neighbourhood of the closure $\overline{\{ u(z): z\in\mathbb{D}(0,R)\}}\subset\mathcal{C}$, with no pole at $u(0)$. Denote by $\sum_{n\in\mathbb{N}} a_n z^n$ the power series expansion of the complex valued function $g\circ u$ near $z=0$. Assume that
	\begin{enumerate}[label=\roman*)]
		\item $\lambda\circ u=id_{\mathbb{D}(0,R)}$;
		\item $u$ extends by continuity into a continuous map $\overline{\mathbb{D}(0,R)}\to\mathcal{C}$;
		\item The first derivative of $\lambda$ at $p=u(R)$ is zero, over $T_{u(R)}\mathcal{C}$: $D_p\lambda=0$;
		\item The second derivative of $\lambda$ at $p$ is non-zero 
			: $D^2_p \lambda\neq 0$;
		\item $u$ extends holomorphically to the neighbourhood of any point in the boundary of $\mathbb{D}(0,R)$, minus the $d$-th roots of $R^d$: $$\partial\mathbb{D}(0,R)\setminus\left\lbrace R(e^{2i\pi/d})^k: k\in \{0,...,d-1\}\right\rbrace;$$
		\item The set of poles of $g$ is the $A$-orbit of $p$. And these poles have multiplicity $M\geq 1$.
		\item $Fix(A^1)=...=Fix(A^{d-1})=\{u(0)\}$ and for any $ z\in\mathbb{D}(0,R)$, $$u(e^{2i\pi/d}z)=Au(z);$$
		\item There is an integer $r\in\{0,...,d-1\}$ such that for any $q\in \overline{\{ u(z): z\in\mathbb{D}(0,R)\}}$, $$g(Aq)=e^{2i\pi r/d}g(q).$$
	\end{enumerate}
	 Then for any non-negative integer $n$ such that $n\not\equiv r \, [d]$ we have $a_n=0$ and there exists a non-zero constant $D$ and a unique sequence of constants $(d_l)_{l\geq 1}$ such that for any positive integer $K\in\mathbb{N}\setminus\lbrace 0\rbrace$ we have the asymptotic expansion:
	\begin{equation}
		a_{dn+r}=DR^{-dn}\frac{1}{n^{3/2-M}}\left(1 + \frac{d_1}{n^{1}}+...+\frac{d_K}{n^{K}}+O\left(\frac{1}{n^{K+1}}\right)\right).
	\end{equation}

\end{Cor}
\begin{proof}
We first deal with the case when $g$ is $A$-invariant. Since $g$ is $A$-invariant, it factors through $\Pi$. Let $g_A:\mathcal{C}_A\to\mathbb{C}$ be its factors through $\Pi$, i.e be such that $g_A\circ \Pi = g$. And $g_A$ admits an only pole at $\Pi(p)$, which has degree $M$. Note that $g_A$ is not the lowering of $g$. Let $u_A$ be the lowering of $u$, as described in \ref{Prop_Lowering} and let $h_A$ be a $\mathcal{C}_A$-valued map that is holomorphic in a neighbourhood of $\overline{\mathbb{D}(1,1)^{1/2}}$, with non-zero derivative at $0$, such that
\begin{equation*}
	u_A(\omega)=h_A\left(\sqrt{1-\frac{\omega}{R^d}}\right).
\end{equation*}
We have justified the existence of such map in the two claims following (\ref{formula_t2}).

We thus have for any complex number $z\in\mathbb{D}(0,R)$, $$g\circ u (z)= g_A\circ\Pi\circ u (z) = g_A\circ u_A (z^d).$$
We thus get that for any complex number $z\in\mathbb{D}(0,R)$,
\begin{equation}
	(g\circ u)(z)= (g_A\circ h_A)\left(\sqrt{1-\left(\frac{z}{R}\right)^d}\right).
\end{equation}

Besides, since the derivative of $h_A$ at $0\in\mathbb{C}$ is non-zero and the function $g_A$ admits a pole of degree $M$, at $h_A(0)=u_A(R)=\Pi(p)$, the function $(g_A\circ h_A)$ admits a pole at $0$ of degree $M$.

It follows from Corollary \ref{Cor_Black_Box_Pole_aperiodic} that if $\sum_n b_n \omega^n$ is the power series expansion of $ \omega\mapsto (g_A\circ h_A)(\sqrt{1-\omega})$ in the neighbourhood of $\omega=0$, then we obtain a unique sequence of constants $(d_l)_{l\geq 1}$ and a constant $D$ such that for any positive integer $K\in\mathbb{N}\setminus\lbrace 0\rbrace$:
	\begin{equation}
		b_n=\frac{D}{n^{3/2-M}}\left(1+\frac{d_1}{n^{1}}+...+\frac{d_{K}}{n^{K}}+ O\left(\frac{1}{n^{K+1}}\right)\right).
	\end{equation}
		
	From this we deduce that the coefficients $(a_n)_n$ of the power series expansion of $g\circ u$ in the neighbourhood of $z=0$, satisfies for any positive integer $K\in\mathbb{N}\setminus\lbrace 0\rbrace$:
	\begin{equation}
		a_{dn}=DR^{-dn}\frac{1}{n^{3/2-M}}\left(1+\frac{d_1}{n^{1}}+...+\frac{d_{K}}{n^{K}}+ O\left(\frac{1}{n^{K+1}}\right)\right)
		\end{equation}
		and for any $n$ such that $n\not\equiv 0 [d]$, we have $a_n=0$. This is what we wanted when $g$ is $A$-invariant. The general case follows from the previous paragraph and the remark \ref{Rem_Reduction}. 
\end{proof}

An easier case than the one described above is when the map $g$ possesses a pole over the $A$-orbit of some point $u(R')$ for $0<R'<R$. In that case we obtain the following Tauberian theorem and its corollary:

\begin{Thm}\label{Asymptotic_Expansion_for_merop_func}
	Let $d$ and $M$ be two positive integers, let $R>R'>0$ be real positive numbers. Consider $g$ a meromorphic function in a neighbourhood of $\overline{\mathbb{D}}(0,R)$ such that
	\begin{enumerate}[label=\roman*)]
	\item The set of poles of $g$ in $\overline{\mathbb{D}}(0,R)$ is $\{R',R'e^{2i\pi/d},... , R' (e^{2i\pi/d})^{d-1} \}$, the set of $d$-th root of $R^d$ and all its poles has degree $M$.
	\item There is an integer $r\in\{0,...,d-1\}$ such that for any $q\in \overline{\{ u(z): z\in\mathbb{D}(0,R)\}}$, $$g(e^{2i\pi/d}q)=e^{2i\pi r/d}g(q).$$
	\end{enumerate}
	
	Denote by $\sum_{n\in\mathbb{N}} b_n z^n$ the power series expansion of $g$ near $z=0$, then for any non-negative integer $n$ such that $n\not\equiv r [d]$ we have $b_n=0$ and there exist a non-zero constant $D_{-M}$ and constants $D_{-M+1}$, ..., $D_{-1}$, such that for some $\varepsilon>0$, and for any positive integer $K\in\mathbb{N}\setminus\lbrace 0\rbrace$, we have the asymptotic expansion as $n$ goes to $\infty$,
	\begin{equation}
		b_{dn+r}=\frac{D_{-M}}{n^{1-M}( R')^{dn}}+\frac{D_{-{M+1}}}{n^{2-M}( R')^{dn}}+...+\frac{D_{-1}}{( R')^{dn}} +O\left(\frac{1}{(R+\varepsilon)^{n}}\right).
	\end{equation}
\end{Thm}
\begin{proof}
We prove in detail this theorem in the case when $M=1$ that is, when the poles of $g$ in the set $\overline{\mathbb{D}}(0,R)$ are simple. The general case, which follows the same method is sketched at the end of the proof. Fix $\zeta_d=e^{2i\pi/d}$.

By Lemma \ref{Multiplicity_of_the_zeros} the function $g$ admits $z=0$ as a zero of order $N+r$ with $N\in d\mathbb{N}$. let $g_1$ be the meromorphic function in a neighbourhood of $\overline{\mathbb{D}}(0,R)$ that is non-zero at $z=0$ and that is defined by:
 $$g_1:z\mapsto \frac{g\circ u (z)}{z^{N+k}}.$$
In particular for any complex number $z\in\mathbb{D}(0,R)$, we have $g_1(\zeta_d z ) = g_1(z)$. Therefore if we denote by $c:=Res(g_1, R')$, we get for any integer $k\in\{0,...,d-1\}$,
	$$Res(g_1, R' \zeta_d^k)=\zeta_d^{k} c.$$

	There exists a holomorphic map $g_2$ in a neighbourhood of the set $\overline{\mathbb{D}}(0,R)$ such that for any complex number $z$ in $\overline{\mathbb{D}}(0,R'+\varepsilon)\setminus \{R'\zeta_d^k: k\in\{0,...,d-1\}\}$, for some $\varepsilon>0$ we have
	$$ g_2(\zeta_d z) = g_2(z) $$
	and for any complex number $z\in\mathbb{D}(0,R)$,
	\begin{equation}\label{eq:g2}
	g_1(z)=\sum_{k=0}^{d-1} \frac{c \zeta_d^k}{z-R'\zeta_d^k} + g_2(z).
	\end{equation}
	
Simplifying the above equality for $g_1(z)$ we get

\begin{equation}\label{eq:g1}
g_1(z)= cd(R')^{d-1}\frac{1}{z^d-(R')^d} + g_2(z).
\end{equation}

Now from Cauchy inequality the coefficients of the power series expansion of $g_2$ in a neighbourhood of $z=0$ are $O\left(\frac{1}{(R+\varepsilon)^n}\right)$, for some $\varepsilon >0$.
The result then immediately follows by expanding $\frac{1}{z^d-(R')^d}$.

For the case when $M\geq 2$, following the same strategy we get the equality for any complex number $z\in\mathbb{D}(0,R)$
\begin{equation}
g\circ u(z)=z^{N+r}\left( \sum_{m=1}^M \frac{c_m}{(z^d-(R')^d)^m} + g_2(z)\right)
\end{equation}
with $N\in d\mathbb{N}$, and $c_1,...,c_M$ being constants, $g_2$ being a holomorphic map in a neighbourhood of the set $\overline{\mathbb{D}}(0,R'+\varepsilon)$ for some $\varepsilon>0$,  such that for any complex number $z$ in $\overline{\mathbb{D}}(0,R'+\varepsilon)$ we have
	$$ g_2(\zeta_d z) = g_2(z).$$
	
Now for any positive integer $m$, the power series expansion of $z\mapsto \frac{1}{(z^d-(R')^d)^m}$ is 
$$\frac{(-1)^M}{(R')^{md}}\sum_{n\geq 0}\frac{\Gamma(m+n)}{\Gamma(m)\Gamma(n+1)} \left(\frac{z}{R'}\right)^{dn}.$$ 
And $n\mapsto\frac{\Gamma(m+n)}{\Gamma(m)\Gamma(n+1)}$ is a polynomial function $P(n)$, with variable $n$ of degree $M-1$. The result follows after simplification.
\end{proof}
\begin{Cor}\label{Cor_tauber_pole_before_aperiodic}
Let $R>R'>0$ be real numbers, $d$ be a positive integer, $\mathcal{C}$ be a smooth analytic complex curve (i.e. a Riemann surface or a complex manifold of dimension 1) and $A\in Aut(\mathcal{C})$ of order $d$ (i.e such that $A,...,A^{d-1}\neq Id_\mathcal{C}$, $A^d=Id_\mathcal{C}$). Let $u:\mathbb{D}(0,R)\to\mathcal{C}$ and $\lambda:\mathcal{C}\to\mathbb{C}$ be two holomorphic maps and let $g$ be a meromorphic function in a neighbourhood of the closure $\overline{\{ u(z): z\in\mathbb{D}(0,R')\}}\subset\mathcal{C}$. Denote by $\sum_{n\in\mathbb{N}} b_n z^n$ the power series expansion of the complex valued function $g\circ u$ near $z=0$. Assume that
	\begin{enumerate}[label=\roman*)]
		\item $\lambda\circ u=id_{\mathbb{D}(0,R)}$;
		\item $g$ is a meromorphic function over a neighbourhood of $\overline{\{u(z): z\in\mathbb{D}(0,R')\}}$;
		\item The set of poles of $g$ is the $A$-orbit of the point $u(R')$, and all its poles have degree $M$;
		\item $Fix(A^1)=...=Fix(A^{d-1})=\{u(0)\}$ and for any $ z\in\mathbb{D}(0,R)$, $$u(e^{2i\pi/d}z)=Au(z);$$
		\item There is an integer $r\in\{0,...,d-1\}$, such that for any $q\in \{ u(z): z\in\mathbb{D}(0,R')\}$, $g(Aq)=e^{2i\pi r/d}g(q)$.
	\end{enumerate}
	
	 Then for any non-negative integer $n$ such that $n\not\equiv r [d]$ we have $b_n=0$ and there exist a non-zero constant $D_{-M}$ and constants $D_{-M+1}$, ..., $D_{-1}$, such that for some $\varepsilon>0$ and for any positive integer $K\in\mathbb{N}\setminus\lbrace 0\rbrace$, we have the asymptotic expansion, as $n$ goes to $\infty$,
	\begin{equation}
		b_{dn+r}=\frac{D_{-M}}{n^{1-M} (R')^{dn}}+\frac{D_{-{M+1}}}{n^{2-M} (R')^{dn}}+...+\frac{D_{-1}}{( R')^{dn}} +O\left(\frac{1}{(R'+\varepsilon)^{n}}\right).
	\end{equation}
\end{Cor}
\begin{proof}
The function $u:\mathbb{D}(0,R)\to\mathcal{C}$ is a biholomorphism from $\mathbb{D}(0,R)$ onto its range with inverse $\lambda$. Thus the function $z\mapsto g\circ u(z)$ is  meromorphic on a neighbourhood of the compact set $\overline{\mathbb{D}}(0,R')$ with set of poles $\{R'e^{2i\pi k/d}: k\in\{0,...,d-1\}\}$, and these poles have degree $M$.
We then apply the previous theorem \ref{Asymptotic_Expansion_for_merop_func}. 
\end{proof}

\bibliographystyle{alpha}
\bibliography{sample}

Institut Mathématiques de Bordeaux/ Institut Montpellierain Alexander Grothendieck\\
\indent E-MAIL: gchevalier@protonmail.com
\end{document}